\documentclass{amsart}
\usepackage[normalem]{ulem}
\usepackage{graphicx}
\usepackage{amssymb,amsmath,amsrefs}
\usepackage[colorlinks=true, linkcolor=red, citecolor=blue, urlcolor=blue]{hyperref}
\usepackage{xcolor}
\usepackage{graphicx}
\usepackage{enumerate}
\usepackage{tikz}
\usetikzlibrary{arrows.meta, positioning}

\usepackage{comment}

\usepackage{soul,color}
\usepackage{array}
\usepackage{wrapfig}
\usepackage{tabularx}
\usepackage{cancel}

\theoremstyle{plain}

\newtheorem{mainthm}{Theorem}

\newtheorem{mainclly}[mainthm]{Corollary}
\newtheorem*{conj*}{Conjecture}
\newtheorem*{cor*}{Corollary}
\newtheorem{theorem}{Theorem}[section]

\newtheorem{lemma}[theorem]{Lemma}

\newtheorem{claim}{Claim}

\theoremstyle{definition}
\newtheorem*{def*}{Definition}
\newtheorem{remark}[theorem]{Remark}

\newtheorem{definition}[theorem]{Definition}

\newcommand{\SC}{{\mathcal C}}

\newcommand{\SE}{{\mathcal E}}

\newcommand{\SM}{{\mathcal M}}

\newcommand{\SO}{{\mathcal O}}
\newcommand{\SP}{{\mathcal P}}

\newcommand{\SR}{{\mathcal R}}

\newcommand{\SU}{{\mathcal U}}

\renewcommand{\epsilon}{\varepsilon}

\newcommand{\Z}{\mathbb{Z}}

\newcommand{\R}{\mathbb{R}}
\newcommand{\eps}{\varepsilon}

\makeatletter
\newcommand{\tpitchfork}{
  \vbox{
    \baselineskip\z@skip
    \lineskip-.52ex
    \lineskiplimit\maxdimen
    \m@th
    \ialign{##\crcr\hidewidth\smash{$-$}\hidewidth\crcr$\pitchfork$\crcr}
  }
}
\makeatother

\numberwithin{equation}{section}

\title{On the Abundance of Phase Transitions in Dynamical Systems}

\author{Alexander Arbieto}
\address[A. Arbieto]{
	Universidade Federal do Rio de Janeiro, Instituto de Matem\'atica,
	Av. Athos da Silveira Ramos, 149,
	21941-909, Rio de Janeiro,
}
\email{arbieto@im.ufrj.br}

\author{Walter Britto}
\address[W. Britto]{
	Universidade Federal do Rio de Janeiro, Instituto de Matem\'atica,
	Av. Athos da Silveira Ramos, 149,
	21941-909, Rio de Janeiro,
}
\email{walter\_britto@ufrj.br}

\author{Elias Rego}
\address[E. Rego]{Faculty of Applied Mathematics, AGH University of Krakow, Krakow, Poland.}
\email{rego@agh.edu.pl}


\begin{document}

\begin{abstract}
In this work, we investigate the mechanisms that trigger phase transitions in both discrete-time and continuous-time dynamical systems. We prove a simple topological condition that implies the existence of Hölder continuous potentials displaying phase transitions. As a consequence, we show that phase transitions are typical in several scenarios. Specifically, for $C^1$ diffeomorphisms and $C^1$ vector fields, we show that phase transitions are generic among non-transitive systems. In low dimensions, we conclude that they are generic for non-Anosov systems. Finally, in the $C^0$ setting, we prove that on any compact topological manifold of dimension different from 4, there exists a dense set of homeomorphisms with finite entropy that display phase transitions.
\end{abstract}

\maketitle

\section{Introduction}

The theory of dynamical systems studies the evolution of time-dependent models, with a central goal being the characterization of their long-term behavior under iteration. Typically, two complementary perspectives are considered: the topological, which focuses on concepts like orbit structure, recurrence, expansivity, and topological entropy; and the statistical or measure-theoretic, which is concerned with ideas such as invariant measures, ergodic averages, decay of correlations, and large deviation properties. Given the close connection of dynamical systems to natural phenomena, many time-evolution models originate from physical or probabilistic contexts, for instance maps modeling statistical mechanical systems, stochastic processes, or flows derived from differential equations. Consequently, numerous concepts in dynamical systems are intrinsically rooted in physical ideas. In particular, thermodynamic formalism provides a powerful toolkit by analogy, introducing notions such as topological pressure and equilibrium states to quantify the complexity or ``disorder'' of a system in terms of potentials or ``weights''. One phenomenon inherited from statistical physics is the phase transition. Phase transitions capture abrupt qualitative changes in the statistical description of a system as a control parameter is varied. Intuitively, they should be understood as the mathematical counterpart of familiar physical phenomena, such as melting or magnetization, but re-framed: instead of matter changing state, we observe non-analytic changes in thermodynamic quantities, most notably the topological pressure. 

To improve the readability of this introduction, we shall present some of the basic definitions needed to state our main results, while more technical definitions are postponed to Section~\ref{section_prelim}. Let $X$ be a compact metric space. In our setting, a dynamical system is given either by a homeomorphism $f:X \to X$ or by a continuous flow on $X$, i.e., a continuous map $\psi:\mathbb{R}\times X \to X$ satisfying $\psi(0,x)=x$ and $\psi(s+t,x)=\psi(t,\psi(s,x))$, for every $x\in X$ and $t,s\in \mathbb{R}$. By a \emph{continuous potential} on $X$, we mean a continuous map $\phi: X \to \mathbb{R}$. Let $\mathcal{C}(X,\mathbb{R})$ denote the set of continuous potentials on $X$. We write $\mathcal{M}_f(X)$ and $\mathcal{M}_{\psi}(X)$ for the set of $f$-invariant probability measures and the set of $\psi$-invariant measures, respectively. Similarly, $\mathcal{M}_f^e(X)$ and $\mathcal{M}_\psi^e(X)$ denote the respective sets of ergodic measures for $f$ and $\psi$. Using the celebrated variational principle, the \emph{topological pressure of $f$ with respect to $\phi$} is defined by
\[
P_{\mathrm{top}}(f,\phi) = \sup \left\{ h_\mu(f) + \int \phi \, d\mu \ ; \ \mu \in \mathcal{M}_f(X) \right\},
\]
where $h_\mu(f)$ denotes the metric entropy of $f$ with respect to $\mu$. Analogously, the \emph{topological pressure of $\psi$ with respect to $\phi$} is defined by
\[
P_{\mathrm{top}}(\psi,\phi) = \sup \left\{ h_\mu(\psi) + \int \phi \, d\mu \ ; \ \mu \in \mathcal{M}_\psi(X) \right\}.
\]

When $\phi\equiv 0$, the topological pressure reduces to the topological entropy, one of the most fundamental quantities in dynamics. The connection between thermodynamic formalism and smooth dynamics was established in the seminal works of Sinai, Bowen, and Ruelle for uniformly hyperbolic systems (or Anosov systems). For such systems, the pressure function exhibits ``nice'' regularity properties. The absence of these ``nice'' properties is precisely what mathematicians call a phase transition. Given a continuous potential $\phi$, it induces a map $\Phi:\mathbb{R} \to \mathbb{R}$ defined as $\Phi(t)=P_{\mathrm{top}}(f,t\phi)$. Following \cite{BC}, we define:

\begin{definition}
    We say that $f$ has a phase transition if there is $\phi\in \mathcal{C}(X,\mathbb{R})$ such that the map $\Phi$ is not analytic.
\end{definition}
The concept of phase transition for flows is defined in a completely analogous way. 
Significant efforts have been made to investigate systems displaying phase transitions; see, for instance, \cites{CR1, CR2, CR3, DGR, Lo, PS}. In the recent work \cite{BV24}, it was proven that phase transitions are generic among $C^1$-diffeomorphisms of surfaces that are not Anosov. Specifically, the authors showed that the set of diffeomorphisms exhibiting phase transitions forms a residual subset of $\text{Diff}^1(S)\setminus\mathbb{A}^1(S)$, where $S$ is a compact surface and $\mathbb{A}^1(S)$ denotes the set of Anosov diffeomorphisms. This result establishes a sharp dichotomy: either a surface diffeomorphism is Anosov (and thus lacks phase transitions for certain potentials), or phase transitions are predominant. 

Inspired by \cite{BV24}, this work aims to investigate the mechanisms yielding phase transitions in dynamical systems. We are interested in establishing simple and general criteria for phase transitions that are easy to verify and can be applied to a wide range of scenarios. Our approach is purely topological, aimed at detecting when a system exhibits phase transitions. Indeed, it is built upon the structure of the chain recurrent set, which is the locus that captures any form of recurrence in the system. Let us recall this concept. Our first main result provides a criterion for phase transitions in terms of chain-transitivity:      

\begin{mainthm}\label{thmA}
Let $f:X \to X$ be a homeomorphism and $\psi$ be a continuous flow on $X$. Assume $h_{\mathrm{top}}(f)<\infty$ and $h_{\mathrm{top}}(\psi)<\infty$.
\begin{enumerate}
    \item If $f$ is not chain-transitive, then $f$ admits a H\"older continuous potential exhibiting a phase transition.    
    \item If $\psi$ is not chain-transitive, then $\psi$ admits a H\"older continuous potential exhibiting a phase transition. 
\end{enumerate}
\end{mainthm}

The previous result provides a simple topological criterion for checking the existence of phase transitions. By providing an easily verifiable condition, it allows one to prove the existence of phase transitions for a large class of systems. A natural question that arises is how large the set of systems displaying phase transitions actually is. As a major application of Theorem~\ref{thmA}, one can extend \cite{BV24}*{Theorem A} to a clean, higher-dimensional version. Whenever $X$ denotes a closed Riemannian manifold, we denote by $\text{Diff}^1(X)$ the set of $C^1$-diffeomorphisms of $X$ endowed with the $C^1$-topology.  

\begin{mainthm}\label{thm:generic diffeos}
Let $X$ be a closed Riemannian manifold. There is a $C^1$-generic set $\mathcal{R}\subset \text{Diff}^1(X)$ such that if $f\in \mathcal{R}$, then either $f$ is transitive or $f$ has a H\"older continuous potential exhibiting a phase transition.  
\end{mainthm}

At first glance, this result may not immediately resemble the original statement of \cite[Theorem A]{BV24}. Nevertheless, the dichotomy provided in the statement can be re-casted into a equivalent form for surface diffeomorphims, as illustrated by the following Corollary.  

\begin{mainclly}\label{cor: generic diffeos}
Let $X$ be a closed Riemannian surface. There is a $C^1$-generic set $\mathcal{R}\subset \text{Diff}^1(X)$ such that if $f\in \mathcal{R}$, then either $f$ is an Anosov Diffeomorphism or it admits a H\"older continuous potential exhibiting a phase transition.    
\end{mainclly}

Another setting where Theorem~\ref{thmA} applies is the generic dynamics of smooth flows. A natural question is whether a flow counterpart to Theorem~\ref{thm:generic diffeos} holds true. Nevertheless, one must exercise caution when transposing problems from diffeomorphisms to flows. Indeed, there are phenomena intrinsic to continuous-time systems that are completely absent in the discrete-time case. A prime example is how traditional hyperbolicity is incompatible with the coexistence of singularities and regular orbits exhibiting non-trivial recurrence. There are several examples of robust chaotic dynamics that lie outside the hyperbolic framework. One of the most celebrated cases is the paradigmatic Lorenz attractor, which paved the way for exploring weaker forms of hyperbolicity for flows. Consequently, beyond the scope of Anosov flows, a vast class of systems remains resilient to classical hyperbolic analysis. Therefore, extending any result involving hyperbolicity to a generic flow setting requires careful consideration of the obstacles posed by the lack of hyperbolicity and the presence of singularities. One key advantage of our topological approach is that it bypasses these geometric nuances, thus allowing us to extend Theorem~\ref{thm:generic diffeos} to flows. To precisely state our results, let $X$ be a compact and boundaryless manifold, and let $\mathfrak{X}^1(X)$ denote the space of $C^1$ vector fields on $X$ endowed with the $C^1$-topology. 

\begin{mainthm}\label{thm:generic flows}
Let $X$ be a closed Riemannian manifold. There exists a residual subset $\mathcal{R} \subset \mathfrak{X}^1(X)$ such that if $V \in \mathcal{R}$, then the flow induced by $V$ is either transitive or admits a H\"older continuous potential exhibiting a phase transition.
Furthermore, if $X$ is three-dimensional, there exists a residual subset $\mathcal{R} \subset \mathfrak{X}^1(X)$ such that if $V \in \mathcal{R}$, then the flow induced by $V$ is either a transitive Anosov flow or it admits a H\"older continuous potential exhibiting a phase transition.
\end{mainthm}

Since our methods are purely topological, it is natural to ask whether similar results hold in the continuous regularity setting. Nevertheless, the flexibility enjoyed by the $C^0$-topology makes the dynamics of homeomorphisms significantly wilder. A key difference is that homeomorphisms can exhibit infinite topological entropy, a phenomenon that does not occur for $C^1$-diffeomorphisms. Consequently, the pressure functions of continuous potentials for such homeomorphisms are always infinite. 
The situation is even worse: a striking result by Yano~\cite{Y80} proves that the set of homeomorphisms on compact manifolds with infinite entropy is $C^0$-generic, rendering the search for results analogous to Theorems~\ref{thm:generic diffeos} and~\ref{thm:generic flows} meaningless in the unrestricted $C^0$-setting. Nevertheless, one can still investigate the abundance of phase transitions by restricting the discussion to systems with finite entropy.

Another issue that arises in lower regularity is that the topology of topological manifolds can be far more intricate. This is illustrated by the existence of exotic manifolds, which complicates their classification and makes robust approximation results scarcer without extra assumptions on the ambient manifold. In our final main result, we overcome these difficulties and establish the abundance of phase transitions for finite-entropy homeomorphisms on compact on topological $n$-manifolds for $n\neq 4$. 

\begin{mainthm}\label{thm: dense homeo}
Let $X$ be a compact and boundaryless $n$-dimensional topological manifold. If $n\neq 4$, there is a $C^0$-dense set $\mathcal{U}\subset \text{Homeo}(X)$ such that every $f\in \mathcal{U}$ has finite entropy and admits a H\"older continuous potential with a phase transition.         
\end{mainthm}

The remainder of this text is organized as follows: in Section~\ref{section_prelim}, we recall classical basic concepts and results that will be used to prove our main theorems. Section~\ref{sec: chaintransitive} is devoted to the proof of Theorem~\ref{thmA}. In Section~\ref{sec: everywhere}, we apply our results to prove Theorems \ref{thm:generic diffeos}, \ref{thm:generic flows} and \ref{thm: dense homeo}.

\section{Preliminaries}\label{section_prelim}

In this section we are going to collect the basic concepts and previously known results needed to obtain our main result.
\subsection{Elementary concepts in discrete-time dynamical systems}

Let $X$ be a compact metric space and $f:X\to X$ be a homeomorphism. For any $x\in X$, the {\it orbit} of $x$, the positive orbit of $x$ and the negative orbit of o $x$ under $f$ are, respectively, the  sets $$\mathcal{O}(x) = \{f^n(x) : n \in \mathbb{Z}\}, \mathcal{O}^+(x) = \{f^n(x) : n \geq 0\} \textrm{ and } \mathcal{O}^-(x) = \{f^n(x) : n \leq0\}.$$  A point $x\in X$ is periodic if there is $n>0$ such that $f^n(x) = x$. The minimal such  $n$ is called the period of $x$ and  is denoted by $\pi(x)$.  If $\pi(x)=1$, then $x$ is said to be a fixed point. Let $Fix(f)$ and $Per(f)$ denote the set of fixed and periodic points of $f$, respectively. The critical set of $f$ is defined as $Crit(f)=Fix(f)\cup Per(f)$. We say that $x \in X$ is a \textrm{non-wandering point} if for every neighborhood $U$ of $x$ there is $n>0$ such that $f^n(U) \cap U \neq \emptyset$. The set of non-wandering points of $f$ is denoted by $\Omega(f)$. 

For $x \in X$, the \textit{$\omega$-limit} of $x$ is defined as the set $\omega(x)$ of accumulation points of the   $\SO^+(x)$. Equivalently,  $y\in \omega(x)$ if there is a sequence of positive integers $n_k\to \infty$, such that $f^{n_k}(x)\to y$. We  define the $\alpha$-limit set of $x$, as the set $\alpha_f(x)  = \omega_{f^{-1}}(x)$. If $x\in \omega(x)$, we say that $x$ is a recurrent point. Denote by $\SR(f)$ the set of recurrent points of $f$. Observe that the sets $\omega(x),\alpha(x),\SR(f)$ and $\Omega(f)$ are compact, invariant and it holds $Per(f) \subset \SR(f)\subset \Omega(f)$ and $\omega(x)\cup\alpha(x)\subset\Omega(f)$, for every $x\in X$.
Recall that given a metric space $X$, we say that a subset $\SR\subset X$ is a residual subset if it contains a countable intersection of open and dense subsets of $X$. We say that $f$ is topologically transitive if for any pair of non-empty open sets there is $n>0$ such that $f^{-1}(U)\cap V$. If $X$ has no isolated points, it is equivalent to the existence of a residual subset $\SR\subset X$ such that for every $x\in X$, one has $\omega(x)=X$.

In \cite{C}, Conley introduced a more general type of recurrence called chain-recurrence and which will be central to this work. Let us now recall this concept.
\begin{definition}
Let $f$ be a homeomorphism on a metric space $X$. Given $\epsilon > 0$ and $x, y \in X$, an $\epsilon$-chain from $x$ to $y$ with respect to $f$ is a finite sequence $x = x_0, x_1, \ldots, x_{n-1}, x_n = y$ in $X$ such that
\[
d(f(x_i), x_{i+1}) < \epsilon
\]
for $i = 0, 1, 2, \ldots, n-1$.
\end{definition}

Two points $x,y \in X$ are said to be chain-equivalent if for each $\epsilon > 0$ there exists an $\epsilon$-chain from $x$ to $y$ and there exists an $\epsilon$-chain from $y$ to $x$. If $x \in X$ is chain equivalent to itself, then it is called \textit{ chain-recurrent point}. The set of all chain recurrent points is called \text{ chain-recurrent set}, and is denoted by $CR(f)$. Notice that $$Per(f)\subset \SR(f)\subset \Omega(f)\subset CR(f),$$
and it is well known that, in general, all the inclusions are strict. 
The relation $\sim$ defined by
\[
x \sim y \quad \text{if and only if} \quad x \text{ is chain equivalent to } y
\]
is an equivalence relation on $CR(f)$. The equivalence classes of the  relation $\sim$ are called the chain-recurrent classes of $f$. Given $x\in CR(f)$, we denote by $C(x)$ the chain-recurrent class of $x$.
 One can use $\epsilon$-chains to extend topological transitivity to a more general type of dynamical irreducibility called chain transitivity. Namely, a set $\Lambda \subset X$ is \textit{chain} transitive if $\Lambda$ is nonempty,  closed and invariant set such that for each $x, y \in \Lambda$, and each $\epsilon > 0$  there exists an $\epsilon$-chain from $x$ to $y$.

\subsection{Elementary concepts in continuous-time dynamical systems}
Next we recall the notions of recurrence for flows. Let $\psi$ denote a continuous flow in $X$. For every $x\in X$, we denote the orbit, the positive orbit and the negative orbit of $x$, respectively, as the sets:
$$\mathcal{O}(x) = \{\psi_t(x) : t \in \mathbb{R}\}, \mathcal{O}^+(x) = \{\psi_t(x) : t \geq 0\} \textrm{ and } \mathcal{O}^-(x) = \{\psi_t(x) : t \leq0\}.$$
\begin{remark}
    Note that we use the same notation for the orbits of both flows and homeomorphisms. This will not cause confusion, however, as these two cases are treated separately throughout this work. Similarly, we employ the same notation for the sets $\omega(x)$ and $\alpha(x)$ in both settings.   
\end{remark}
We say that a point $x \in X$ is a \emph{fixed point} (or \emph{singularity}) for $\psi$ if $\psi_t(x) = x$ for every $t \in \mathbb{R}$; otherwise, $x$ is called a \emph{regular point}. A regular point is \emph{periodic} if there exists $t > 0$ such that $\psi_t(x) = x$. The \emph{period} of a periodic point $x$ is defined as $\pi(x) = \inf\{t > 0 \colon \psi_t(x) = x\}$. Let $Sing(\psi)$ and $Per(\psi)$ denote the set of singularities and periodic points of $\psi$, respectively. The critical set of $\psi$ is then defined as $Crit(\psi)=Sing(\psi)\cup Per(\psi)$.

The definitions of the sets $\operatorname{Per}(\psi)$,$\SR(\psi)$, $\omega(x)$, and $\alpha(x)$, as well as the concept of topological transitivity for $\psi$, are verbatim identical to their counterparts for homeomorphisms, simply replacing the discrete time $n \in \mathbb{Z}$ with the continuous time $t \in \mathbb{R}$. A point $x\in X$ is a non-wandering point if for every neighborhood $U$ of $x$ and every $T>0$, there is $t>T$ such that $\psi_t(x)\in U$. Let us denote the set of non-wandering points of $\psi$ by $\Omega(\psi)$.

Next, we recall the concept of chain recurrence for flows. 
\begin{definition}\label{def1}
 Given $\epsilon, T > 0$ and $x, y \in X$, an $\epsilon$-$T$-chain from $x$ to $y$ with respect to $\psi$ is a finite sequence $(x_i,t_i)_{i=0}^{n-1}$  such that $t_i\geq T$ and 
\[
d(\psi_{t_i}(x_i), x_{i+1}) < \epsilon
\]
for $i = 0, 1, 2, \ldots, n-1$.
\end{definition}

Two points $x,y \in X$ are said to be chain-equivalent if for every $\epsilon,T > 0$ there exist an $\epsilon$-$T$-chain from $x$ to $y$ and  an $\epsilon$-$T$-chain from $y$ to $x$. If $x \in X$ is chain equivalent to itself, then it is called \textit{ chain-recurrent point}. The set of all chain recurrent points is called \text{ chain-recurrent set}, and is denoted by $CR(\psi)$. Again, we have $$Per(\psi)\subset \SR(\psi)\subset \Omega(\psi)\subset CR(\psi),$$
and it is well known that, in general, all the inclusions are strict. 
The relation $\sim$ defined by
\[
x \sim y \quad \text{if and only if} \quad x \text{ is chain equivalent to } y
\]
is an equivalence relation on $CR(f)$. Analogously to the discrete-time case, we define an equivalence relation in $CR(\psi)$ and the equivalence class of any point $x$ is denote by $C(x)$. We say that a set $\Lambda \subset X$ is \textit{chain} transitive if for each $x, y \in \Lambda$, and each $\epsilon,T > 0$  there exist an $\epsilon$-$T$-chain from $x$ to $y$.

\subsection{Topological Entropy}

Next, we recall the concept of  topological entropy.  Fix a subset  $N\subset X $ and $\varepsilon, n>0$. A subset set $K\subset N$ is called a  \textit{$n$-$\varepsilon$-separated} subset of $N$ if for any pair of distinct points $x,y\subset K$ there is some $0\leq n_0\leq n$ such that $$d(f^{n_0}(x),f^{n_0}(y))>\varepsilon.$$ 
Let $S(n,\varepsilon, N)$ denotes the maximal cardinality of a $n$-$\varepsilon$-separated subset of $N$.  
It is easy to see that  $S(n,\varepsilon,N)$  is finite due to the compactness of $X$.
 Denote $$h(f,N,\eps)=\limsup\limits_{n\to\infty}\frac{1}{n}\log(S(n,\varepsilon, N)).$$

    We define the \textit{topological entropy} of $f$ on $N$ as the number  $$h_{top}(f,N)=\lim\limits_{\varepsilon\to 0}h(f,N,\eps).$$
    The topological entropy of $f$ is defined as the number $$h_{top}(f)=h_{top}(f,X). $$
it is  well-known that $h_{top}(f^n)=nh_{top}(f)$, for every $n>0$. Next let us recall the concept of entropy-expansiveness. We say that $f$ is \textit{entropy-expansive} if there is $e>0$  such that $h_{top}(f,B^{\infty}_e(x,f))=0$, for every $x\in X$, where
$$B^{\infty}_e(x,f)=\{y\in X; d(f^n(x),f^n(y))\leq e, \textrm{ for every } n\in \Z \}.$$

If $\psi$ is a continuous flow in $X$, we define $$h_{top}(\psi,N)=h_{top}(\psi_1,N)\textrm{ and }h_{top}(\psi)=h_{top}(\psi_1).$$
We then define $$B^{\infty}_e(x,\psi)=\{y\in X; d(\psi_t(x),\psi_t(y))\leq e, \textrm{ for every } t\in \R \},$$
and say that $\psi$ is \textit{entropy-expansive} if there is $e>0$  such that $h_{top}(f,B^{\infty}_e(x,\psi))=0$, for every $x\in X$.

\subsection{Ergodic Theory and Metric Entropy}

Next, we shall recall some well-known concepts from ergodic theory. A borealian probability $\mu$ on $X$ is {\it $f$-invariant}  if $\mu(A)=\mu(f^{-1}(A))$, for every measureable set $A$. We say that an $f$-invariant measure $\mu$ is {\it ergodic} if one has $\mu(A)\cdot\mu(A^c)=0$ for every $f$-invariant set. We denote by $\SM(X)$, $\SM_f(X)$ and $\SM_f^e(X)$ the sets of probability measures of $M$, invariant measures of $f$  and ergodic measures of $f$, respectively. The space of $\SM(X)$ can be made a compact metrizable space by endowing it with the weak*-topology. In this case,  we have $$\mu_n\to \mu \textrm{ if, and only if, } \int\phi d\mu_n\to \int\phi d\mu, \forall\phi\in \SC(X,\R).$$ 
Notice that, by definition,  the map $\mu\to \int \phi d\mu$ is continuous in $\SM(X)$. The celebrated Krylov–Bogolyubov's Theorem states that $\SM_f(M)$ is non-empty, whenever $f$ is a continuous map. In addition, by the Ergodic Decomposition Theorem the set $\SM_f^e(X)$ is also non-empty. 

For each $x\in M$, define the $n$-th empirical measure of $x$ as 
$$\SE_n(x)=\frac{1}{x}\sum_{j=0}^{n-1}\delta_{f^j(x)},$$
where $\delta_y$ denotes the Dirac's measure at $y$. We say that $x$ is a {\it generic point of $\mu\in \SM_f(X)$}, if $\SE_n(x)\to \mu.$

 Let $\SP$ be a finite measurable partition of $X$. For any $n>0$, we denote $$\mathcal{P}_n=\mathcal{P}\vee f^{-1}(\mathcal{P})\vee\cdots\vee f^{n-1}(\mathcal{P}).$$
The metric entropy of $\SP$ is defined as: $$ h_{\mu}(f,\mathcal{P})=-\lim_{n\to\infty}\frac{1}{n}\sum_{P\in\mathcal{P}_n}\mu(P)\log\mu(P).$$
Finally,  the \textit{metric entropy} of $f$ with respect to $\mu$ is given by 
\begin{displaymath}
    h_{\mu}(f)=\sup\lbrace h_{\mu}(f,\mathcal{P}): \mathcal{P}\text{ is a finite partition}\rbrace, 
\end{displaymath}
The metric and topological entropies of $f$ are then linked by the celebrated variational principle:
For any continuous map $f:M\to M$, it holds
$$h_{top}(f)=\sup\{h_{\mu}(f); \mu\in \SM_f(X)\}=\sup\{h_{\mu}(f); \mu\in \SM_f^e(X)\}.$$

\section{Chain-Transitivity and Phase Transitions}\label{sec: chaintransitive}
This section is devoted to the proof of Theorem~\ref{thmA}. Since its statement consists of two parts (one concerning discrete-time systems and the other dealing with continuous-time systems) and the arguments in both scenarios follow exactly the same lines, we shall provide a detailed proof only for the case of homeomorphisms. For the flow setting, we will outline how to adapt the steps.  

\subsection{The discrete-time case}
Let $f:X\to X$ be a homeomorphism.  We begin by recalling the concept of  attracting sets and repelling sets. An open set $U\subset X$ is a \textit{ trapping region } for a $f$ if $f(\overline{U})\subset U$.
 A set $\Lambda$ is an \textit{ attracting set} for $f$ if there exists a trapping region $U$ for $f$ such that
\[
\Lambda = \bigcap_{n \geq 0} f^n(U).
\]
A set $\Lambda^*$ is a \textit{ repelling set} for $f$ if there exists a trapping region $U$ for $f$ such that
\[
\Lambda^* = \bigcap_{n \leq 0} f^n(X \setminus U).
\]

When we wish to emphasize the dependence of an attracting set $\Lambda$ or a repelling set $\Lambda^*$ on the trapping region $U$ from which it arises, we denote it by $\Lambda_U$ or $\Lambda_U^*$, respectively. A pair $(\Lambda, \Lambda^*)$ of subsets of $X$ is an \textit{ attracting-repelling pair} for $f$ if there exists a trapping region $U$  such that
$$
\Lambda = \bigcap_{n \geq 0} f^n(U) \quad \text{and} \quad \Lambda^* = \bigcap_{n \leq 0}f^n(X \setminus U)
$$

It is well known that if $U$ is a trapping region for a homeomorphism $f$ associated with the attracting-repelling pair $(\Lambda, \Lambda^*)$, then $X \setminus \overline{U}$ is a trapping region for $f^{-1}$. Moreover, $(\Lambda^*, \Lambda)$ is the attracting-repelling pair for $f^{-1}$ associated with the trapping region $X \setminus \overline{U}$. Attracting and repelling sets capture the asymptotic behavior of every orbit of a homeomorphism. Precisely, if $(\Lambda, \Lambda^*)$ is an attracting-repelling pair and $x \in M \setminus \Lambda^*$, then $\omega(x) \subset \Lambda$. Conversely, if $x \in M \setminus \Lambda$, then $\alpha(x) \subset \Lambda^*$.

Conley \cite{C} showed that given a trapping region $U$, its intersection with the chain recurrent set $CR(f)$ necessarily lies within the associated attracting set. In other words, if $U$ is a trapping region for $f$, then
\[
CR(f) \cap U \subseteq \Lambda.
\]
Since the non-wandering set satisfies $\Omega(f) \subset CR(f)$, it follows that $\Omega(f) \cap U \subseteq \Lambda$. Furthermore, since $CR(f) = CR(f^{-1})$, one also has that
\[
\Omega(f) \cap (CR(f) \setminus U) \subseteq \Lambda^*.
\]

We are now ready to present our main result. The key idea of the proof is to show that, under our hypotheses, one can guarantee the existence of a non-trivial trapping region $U$. From this region, we construct a suitable Hölder continuous potential based on the attracting-repelling pair $(\Lambda, \Lambda^*)$ associated with $U$, which ultimately allows us to prove that the system exhibits a phase transition.  The next lemma reduces the condition of not being chain-transitive to an analogous condition induced in $CR(f)$. This reduction will be applied many times throughout this work.

\begin{lemma}\label{lemma: non-transitive}
Let $f:X\to X$ be a homeomorphism and $\phi$ a continuous flow on $X$. Then $f$ is not chain-transitive if and only if $CR(f)$ is not chain-transitive. Similarly,  $\phi$ is not chain-transitive if and only if $CR(\phi)$ is not chain-transitive.    
\end{lemma}

\begin{proof}
The proofs for homeomorphisms and flows are completely analogous. For this reason, we will only provide the proof for the discrete-time case. 
Naturally, we only need to prove the sufficiency. Suppose that $f$ is not chain-transitive, but $CR(f)$ is chain-transitive. Hence, the set $U=X\setminus CR(f)$ is non-empty and open. Fix $y\in U$ and $\varepsilon>0$. Observe that $\alpha(y)\cup \omega(y)\subset CR(f)$ and, hence, there exist $n_1,n_2>0$ such that 
$$d(f^{n_1}(y),CR(f))<\varepsilon \quad \text{and} \quad d(f^{-n_2}(y),CR(f))<\varepsilon.$$

Let $x_1,x_2\in CR(f)$ be two points realizing the previous distances, respectively. Since $CR(f)$ is chain-transitive, we obtain an $\varepsilon$-chain $(z_i)_{i=0}^k$ connecting $x_1$ to $x_2$. Consequently, the sequence of points 
$$y, f(y), \dots, f^{n_1-1}(y), z_0, \dots, z_{k-1}, f^{-n_2}(y), \dots, f^{-1}(y), y$$
forms a $2\varepsilon$-chain from $y$ to $y$. Since $\varepsilon$ was chosen arbitrarily, we conclude that $y\in CR(f)$, contradicting the choice of $y\in U$. Consequently, $CR(f)$ is not chain-transitive.   
\end{proof}

\begin{proof}[Proof of item (1) in Theorem \ref{thmA}]
Suppose $f$ is not chain transitive. Then, Lemma \ref{lemma: non-transitive} implies $CR(f)$ is not chain-transitive. Consequently, there  are points $y,z \in CR(f)$ and $\varepsilon> 0$ such that do not exist an $\varepsilon$-chain from $y$ to $z$. Now consider the set 
\[
U = \left\{ x \in X \,\middle|\,  \text{ there exists an } \varepsilon\textrm{-chain from } y \text{ to } x \right\}.
\]

\begin{claim}
    $U$ is a non-empty trapping region for $f$ which is not equal to $X$.
\end{claim}

\begin{proof}
Firstly, observe that $y\in U$ and, by definition, $z\not\in U$. So $U\neq X$. The invariance of $U$ is straightforward. Indeed, if $x\in U$, there there is an $\eps$-chain $$x_0=y,x_1,\cdots,x_n=x.$$ By considering $$x_0=y,x_1\cdots , x_n=x,x_{x+1}=f(x),$$ we see that $f(x)\in U$. Next, we show that $U$ is open. For this sake, fix $x\in U$ and let $$x_0=y,x_1,\cdots,x_{n-1},x_n=x$$ be an $\eps$-chain. It is clear that $x\in B_{\eps}(f(x_{n-1}))$ and $B_{\eps}(f(x_{n-1}))\subset U$.

Finally, consider $x\in \overline{U}$. By continuity, we can fix $\delta>0$ such that if $d(v,w)\leq \delta$, then $d(f(v),f(w))<\eps$. Chose any point $w\in U\cap B_{\frac{\delta}{2}}(x)$. Since $w\in U$, there is an $\eps$-chain $$x_0=y,x_1,\cdots,x_{n-1},x_n=w.$$ By the choice of $\delta$, we get $d(f(x),f(w))\leq \eps$. In particular, the sequence $$x_0=y,x_1,\cdots,x_{n-1},w,f(x)$$ is an $\eps$-chain from $y$ to $f(x)$. This shows $f(\overline{U})\subset U$ and the claim is proved. 
\end{proof}

Now, since $\Omega(f) \subseteq CR(f)$, we have that $\Omega(f) \cap U \subseteq \Lambda_U.$ We can then fix  $$\Omega_2(f) = \Omega(f) \backslash \Omega_1(f) = \Omega(f) \cap (CR(f) \backslash U).$$ Hence,  $$\Omega_2(f) = \Omega(f) \backslash \Omega_1(f) = \Omega(f) \cap (CR(f) \backslash U) \subset \Lambda^*_U.$$

Next, we construct a Hölder continuous potential $\phi$ on $X$. Let $\delta_1 = d(\Lambda, \partial U) > 0$. Since the repeller $\Lambda^* \subset X \setminus \overline{U}$ is compact and $f^{-1}(X \setminus U) \subset \operatorname{int}(X \setminus U)$, we can find a neighborhood $V$ of $\Lambda^*$ such that $\overline{V} \cap \overline{U} = \emptyset$ and $\delta_2 = d(\Lambda^*, \partial V) > 0$. Moreover, we can choose $V$ small enough so that the $\delta_1$-neighborhood of $\Lambda$ and the $\delta_2$-neighborhood of $\Lambda^*$ are disjoint. Setting $\delta = \min\{\delta_1, \delta_2\} > 0$, we define
$$
\phi_1(x) = \max \bigg\{0,1 - \frac{d(x,\Lambda_U)}{\delta}\bigg\}\ \quad \text{and} \quad \phi_2(x) = \max \bigg\{0,1 - \frac{d(x,\Lambda^*_U)}{\delta}\bigg\}
$$
and  $\phi(x) = \phi_1(x) - \phi_2(x)$.

By construction, $\phi_1$ is supported in the $\delta$-neighborhood of $\Lambda$, while $\phi_2$ is supported in the $\delta$-neighborhood of $\Lambda^*$, and these supports are disjoint. Thus, $\phi(x) = 1$ on $\Lambda$, $\phi(x) = -1$ on $\Lambda^*$, and $|\phi(x)| < 1$ elsewhere. The function $\phi$ is Hölder continuous since it is a combination of Lipschitz functions. Recall that any invariant measure $\mu$ satisfies $\mu(\Omega(f)) = 1$. Since every point outside the attractor-repeller pair $(\Lambda, \Lambda^*)$ is wandering, it follows that no invariant measure can be supported on $CR(f) \setminus (\Lambda \cup \Lambda^*)$; that is, $\mu(CR(f) \setminus (\Lambda \cup \Lambda^*)) = 0$. Therefore, points outside $\Lambda \cup \Lambda^*$ do not contribute to the pressure function.

Now, let $\mu \in \mathcal{M}_f^e(X)$ be an ergodic measure. If $x \in U$ is a generic point for $\mu$, then its forward orbit converges to $\Lambda$, and hence $\operatorname{supp}(\mu) \subset \Lambda$. Analogously, if $x \in X \setminus U$ is a generic point for $\mu$, then the backward orbit of $x$ converges to $\Lambda^*$, implying that $\operatorname{supp}(\mu) \subset \Lambda^*$. Consequently, the metric pressure function $P_\mu(f, t\phi) := h_\mu(f) + t \int \phi \, d\mu$ takes the form
$$
     P_{\mu}(f,t\phi) = h_{\mu}(f) + t \int \phi \, d\mu = \begin{cases} 
     h_\mu(f) + t, & \text{supp}(\mu) \subset \Lambda_U, \\
     h_\mu(f) - t, & \text{supp}(\mu) \subset \Lambda_U^*
     \end{cases}
$$

By assumption, $h_{top}(f)<\infty$. So, the variational principle implies the existence of  $0\leq a,b< \infty$ satisfying
$$
a := \sup_{\mu \in \SM_f^e(X)} \{h_\mu(f) \mid \text{supp}(\mu) \subset \Lambda_U\} \quad \text{and} \quad b := \sup_{\mu \in \SM_f^e(X)} \{h_\mu(f) \mid \text{supp}(\mu) \subset \Lambda_U^*\}. 
$$
Then one has that
\begin{align*}
P_{\text{top}}(f,t\phi) &= \sup_{\mu \in \SM_f^e(X)} \left\{ h_\mu(f) + t \int \phi  d\mu \right\} \\
&= \max \left\{ \sup_{\text{supp$(\mu) \subset \Lambda_U$}} h_\mu(f) + t, \sup_{\text{supp$(\mu) \subset \Lambda_U^*$}} h_\mu(f) - t \right\} \\
&= \max \left\{ a + t, b - t \right\}
\end{align*}
which is not differentiable at $t_0 = \frac{b - a}{2}$ (or in $t_0 = 0$ in the case of $a=b$). Therefore, a phase transition occurs.
\end{proof}

\subsection{The flow case}
In this subsection, we explain how to adapt the proof of item~(1) in Theorem~\ref{thmA} to the context of flows. The core strategy in the homeomorphism case was to leverage the fact that $CR(f)$ is not chain-transitive to construct an attracting-repelling pair for $f$, which was subsequently used to build a Hölder continuous potential exhibiting a phase transition. The construction of this pair was established in Claim~1.

For the case of flows, the theories of chain recurrence and attracting-repelling pairs translate seamlessly from the discrete-time setting (see, for instance, \cite{AL} for a comprehensive exposition of recurrence in topological flows). This dictionary preserves all the essential properties used in item~(1). Assuming $CR(\psi)$ is not chain-transitive, there exist $y,z \in CR(\psi)$ and $\varepsilon, T > 0$ such that no $\varepsilon$-$T$-chain connects $y$ to $z$. In this continuous-time setting, the trapping region introduced in Claim~1 takes the form
\[
U = \{x \in X \colon \text{there exists an } \varepsilon\text{-}T\text{-chain from } y \text{ to } x\}.
\]
The proof of Claim~1 in this scenario is a direct adaptation of the discrete-time argument. Once the attracting-repelling pair is obtained, the construction of the potential function likewise follows mutatis mutandis, yielding item~(2) and thereby completing the proof of Theorem~\ref{thmA}.

\section{Phase Transitions Everywhere}\label{sec: everywhere}
In this section we shall see that phase transitions are abundant in the sense that they are present in large sets of dynamical systems. We are particuarly interesed in proving Theorems \ref{thm:generic diffeos}, \ref{thm: dense homeo} and Corollary \ref{cor: generic diffeos}. 

\subsection{\texorpdfstring{$C^1$}{C1}-generic diffeomorphisms}
We begin by turning our attention to attention to Theorem \ref{thm:generic diffeos}. Before starting with its proof, we first need to setup some notation related to the setting of diffeomorphisms. Let us consider $f:X\to X$ a $C^1$ diffeomorphism of a closed manifold $X$. 
We say that a compact and invariant set $\Lambda\subset X$ is a \textit{ hyperbolic set} if there are $C>0$, $0<\lambda<1$ and a $Df$-invariant splitting $T\Lambda=E^{s}\oplus E^u$ satisfying:
\begin{itemize}
    \item $\Vert Df^n|_{E^s}  \Vert\leq C\lambda^ n$, for every $x\in \Lambda$ and $n\geq 0$.
    \item  $\vert Df^{-n}|_{E^u_x}\Vert\leq C\lambda^{n}$, for every $x\in \Lambda$ and $n\leq 0$.
    \end{itemize}
We say that a periodic point is \textit{hyperbolic} if $\SO(p)$ is a hyperbolic set. By the invariant manifold theory (see \cite{HPS}), for every point $x$ in a hyperbolic set $\Lambda$, the bundles $E_x^s$ and $E_x^u$ are, respectively,  uniquely integrable to the $C^1$-manifolds:
\[W^{s}(x)=\{y\in X; d(f^n(x),f^n(y))\to 0, n\to \infty \}   \]
\[\textrm{ and }\] 
\[W^{u}(x)=\{y\in X; d(f^n(x),f^n(y))\to 0, n\to -\infty \}.   \]
If $p$ is a hyperbolic periodic point for $f$, denote $$W^s(\SO(p))=\bigcup_{z\in \SO(p)} W^s(z) \textrm{ and } W^u(\SO(p))=\bigcup_{z\in \SO(p)} W^u(z).$$

We say that $\Lambda$ is a \textit{homoclinic class} if there is a hyperbolic periodic point $p\in \Lambda$ such that $$\Lambda=\overline{W^s(\SO(p))\pitchfork W^s(\SO(p)) }.$$
We are now in position a to obtain a proof for Theorem \ref{thm:generic diffeos}.

\begin{proof}[Proof of Theorem \ref{thm:generic diffeos}]
To begin with, recall that every $C^1$-diffeomorphism on a compact manifold has finite topological entropy. According to \cite{C06}, there exists a $C^1$-generic set $\mathcal{R} \subset \operatorname{Diff}^1(X)$ such that if $f \in \mathcal{R}$, then every isolated chain-recurrent class of $f$ is a homoclinic class. In particular, if $f$ is chain-transitive, then $CR(f) = X$ is the unique chain-recurrent class of $f$, which is trivially isolated. Thus, we conclude that $CR(f)$ is indeed a homoclinic class, which implies that $f$ is topologically transitive. In other words, for every $f \in \mathcal{R}$,  transitivity and chain-transitivity are equivalent properties.

Now,  there are only two possibilities for $f \in \mathcal{R}$: either $f$ is transitive, or $f$ is not chain-transitive. Consequently, by Theorem \ref{thmA}, either $f$ is transitive or $f$ admits a H\"{o}lder continuous potential exhibiting a phase transition. This completes the proof.
\end{proof}

\begin{proof}[Proof of Theorem \ref{cor: generic diffeos}]
   First, assume $\dim(X)=2$. It is well known that any $C^1$-diffeomorphism on a compact manifold has finite topological entropy. Now, Ma\~n\'e proved in \cite{M82} that there exists a $C^1$-generic set $\mathcal{R'}\subset Diff^1(X)$ such that if $f\in \SU$, then it is either an Axiom A diffeomorphism or it exhibits infinitely many sinks or sources. Let $\SR$ be given by the intersection of $\SR'$and the residual set given in Theorem \ref{thm:generic diffeos}. Since transitive diffeomorphisms cannot have sources or sinks, any transitive diffeomorphism $f\in \SR$ must by Axiom A and hence Anosov. Consequently, Theorem \ref{thmA} gives that any $f\in \SR$ is either Anosov or has phase transitions, thereby completing the proof of item (1).

\end{proof}

\subsection{Phase Transitions for \texorpdfstring{$C^1$}{C1}-Generic  flows}\label{sec: flows}
In this subsection, we will explore phase transitions in the setting of differentiable flows. We are going to   apply our findings to show that phase transitions are actually a frequent property along the set of $C^1$-vectors fields.  Our main goal here is to provide the reader with proof for Theorem \ref{thm:generic flows}.  Before to proceed with the proof, we need to  recall some fundamental concepts from the theory of differentiable flows.  

Hereafter, $M$ denotes a closed Riemanninan manifold and $\mathfrak{X}^1(M)$ denotes the set of $C^1$-vector fields over $M$ endowed with the $C^1$-topology. We assume that $d$ is the metric induced on $M$ by its Riemannian metric. For any $V\in \mathfrak{X}^1(M)$, let $\psi$ be the flow generated by $V$. Observe that, in this case, $\sigma$ is a singularity of $\psi$, if and only if, $V(\sigma)=0$.  For every compact and $\psi$-invariant subset $\Lambda$  the stable and unstable sets of $\Lambda$ are respectively defined as:
$$
    W^s(\Lambda)= \{x \in M : \omega(x) \subset \Lambda\} \quad \text{ and } 
    W^u(\Lambda)= \{q \in M : \alpha(x) \subset \Lambda\}.$$

 An \emph{attractor} of $\psi$ is a transitive attracting set $\Lambda$, whereas a \emph{repeller} is an attractor for the flow generated by $-V$. We say that $\Lambda$ is a \emph{sink} of $\psi$ if it is a trivial attractor, namely, if it reduces to a single orbit of $\psi$. Conversely, $\Lambda$ is a \emph{source} of $\psi$ if it is a trivial repeller. For an attractor $\Lambda$, the set $W^s(\Lambda)$ defined above is called its \emph{basin of attraction}, which consists of all points whose future orbits converge to $\Lambda$. An attracting set whose basin of attraction is the entire manifold is called a \emph{global attractor}.
    A compact invariant set $\Lambda$ is \emph{hyperbolic} if it exhibits $D\psi$-invaranti splitting $T_\Lambda M = E^s \oplus \langle V \rangle \oplus E^u$ and there are constants $\lambda, C > 0$ such that:
\begin{itemize}
    \item $\|DX_t(x)|_{E^s_x}\| \leq C e^{-\lambda t}$, $\forall x \in \Lambda$, $\forall t \geq 0$.
    \item $\|DX_{-t}(x)|_{E^u_x}\| \leq C e^{-\lambda t}$, $\forall x \in \Lambda$, $\forall t \geq 0$.
\end{itemize}

A singularity or a periodic orbit of $X$ is \emph{hyperbolic} if its orbit is a hyperbolic set of $X$. 
A hyperbolic set $\Lambda$ of $X$ is called \emph{basic} if it is transitive and \emph{isolated}, namely $\Lambda = \bigcap_{t \in \mathbb{R}} X_t(U)$ for some neighborhood $U$ of $\Lambda$. By invariant manifold theory\label{TVE} \cite{HPS} it follows that for every $x \in \Lambda$ the sets
\begin{align*}
    W^{ss}(x) &= \{y \in M : d(\psi_t(x), \psi_t(y)) \to 0, t \to \infty\} \quad \text{and} \\
    W^{uu}(x) &= \{y \in M : d(\psi_t(x), \psi_t(y)) \to 0, t \to -\infty\}
\end{align*}
are invariant $C^1$-manifolds tangent to $E^s_x$ and $E^u_x$ respectively.  For any  $x\in \Lambda$, it holds 
\begin{align*}
    W^{s}(\mathcal{O}(x)) &= \bigcup_{t \in \mathbb{R}} W^{ss}(\psi_t(x)) \quad \text{and} \\
    W^{u}(\mathcal{O}(x)) &= \bigcup_{t \in \mathbb{R}} W^{uu}(\psi_t(x)).
\end{align*}
Furtherore, these sets are invariant manifolds tangent to $E^s_x \oplus \langle V(x) \rangle$ and $\langle V(x) \rangle \oplus E^u_x$, respectively. 
Note that $W^{ss}(\sigma) = W^{s}(\sigma)$ and $W^{uu}(\sigma) = W^{u}(\sigma)$ for every hyperbolic singularity $\sigma$ of $X$. 
We say that a compact and invariant set $\Lambda\subset M$ is a \textit{homoclinic class} if there exists a hyperbolic periodic point $p\in M$ such that $$\Lambda=\overline{W^s(\SO(p))\pitchfork W^s(\SO(p))}.$$

 Unlike the case of surface diffeomorphisms, the possible presence of singularities prevents us from obtaining the generic set constructed in \cite{M82}. Therefore, we shall adopt a different approach. 
A flow $\psi$ is said to satisfy \emph{Axiom~A} if its nonwandering set $\Omega(\psi)$ is hyperbolic and
$
\Omega(\psi)=\overline{\operatorname{Crit}(\psi)}.
$
The celebrated Spectral Decomposition Theorem asserts that if $\psi$ satisfies Axiom~A, then $\Omega(\psi)$ admits a disjoint decomposition
$
\Omega(\psi)=\Lambda_1\cup\cdots\cup\Lambda_k,
$
where each $\Lambda_i$ is a basic set of $\psi$, that is, each $\Lambda_i$ is a transitive hyperbolic set.

Next, we recall the concept of singular Hyperbolic set.
Let $\Lambda$ be an invariant compact set. An $D\psi$-invariant  splitting $T_{\Lambda}M = E \oplus F$ is \emph{dominated} if there exist constants $C, \lambda > 0$ such that for any $t > 0$, and any $x \in \Lambda$, it holds 
\[\frac{||D\psi_t(x)|_{E_x}||}{m(D\psi_t(x)|_{F_x})} \leq Ce^{-\lambda t}\]
where $m(T) = \inf_{\|v\|=1} \|T(v)\|$ the minimum norm of a linear operator $T$.

We say that $\Lambda$ is singular hyperbolic if all of its singularities are hyperbolic and there is a dominated splitting $T_{\Lambda}M = E^{ss} \oplus E^{cu}$ with respect to $D\psi_t$ and constants $C, \lambda > 0$ such that for any $t \geq 0$, and for any $x \in \Lambda$ :
\begin{itemize}
\item $\|D\psi_t(x)|_{E^{ss}(x)}\| \leq Ce^{-\lambda t}$ 
    \item $|\text{det} (D\psi_{t}(x)|_{E^{cu}(x)})| \geq Ce^{\lambda t}$ 
\end{itemize}

An important fact about singular-hyperbolic sets is that they properly extends the concept of hyperbolic set, as evidenced by the famous hyperbolic lemma:

\begin{theorem}\label{hyperboliclemma}
   Let $\psi$ be a $C^1$-flow on a three-dimensional manifold $M$. If $\Lambda\subset M$ is  singular hyperbolic set without singularities, then $\Lambda$ is hyperbolic. 
\end{theorem}

We now recall the singular version of Axiom A flows.

\begin{definition}
    We say that $\psi$ is \emph{singular Axiom A} if there is a finite disjoint decomposition $\Omega(\psi) = \Lambda_1 \cup \cdots \cup \Lambda_n$ such that where each $\Lambda_i$ is eiter:
    \begin{enumerate}
        \item a hyperbolic basic set,
        \item a singular hyperbolic attractor or
        \item a singular hyperbolic repeller
    \end{enumerate}
    
\end{definition}
\begin{remark}
    It follows from the spectral decomposition theorem  above that every Axiom A flow is  singular Axiom A.
\end{remark}

\begin{definition}
Suppose $\psi$ is a singular Axiom A flow and denote $\Omega(\psi)=\Lambda_1\cup\cdots\cup \Lambda_n$. we say that $\psi$ has a  \emph{cycle} if there are $1\leq i_1<...<i_k\leq n$ such that, for all $j = 1, \dots, k-1$ it holds:  $$W^u(\Lambda_{i_j}) \cap W^s(\Lambda_{i_{j+1}}) \neq \emptyset  \textrm{ and } W^u(\Lambda_{i_k}) \cap W^s(\Lambda_{i_{1}}) \neq \emptyset. $$ Otherwise, we say that $\psi$ is without cycles.
\end{definition}

Our last ingredient is the following result. 
\begin{theorem}[\cite{MP03}]\label{thm_residual}
Let $M$ be a three-dimensional manifold. There is a residual ser $\SR\subset \mathfrak{X}^1(M)$ such that if $V\in \SR$ and $\psi$ is the flow generated by $V$, then either:
\begin{enumerate}
    \item $\psi$ has infinitely many sinks or sources or
    \item $\psi$ is singular Axiom A without cycles.
\end{enumerate} 
\end{theorem}

We are finally in a position of proving Theorem \ref{thm:generic flows}.

\begin{proof}[Proof of  Theorem \ref{thm:generic flows}]
The proof for the first part follows the same lines as the proof of Theorem \ref{thm:generic diffeos}. The crucial step is to show that $C^1$-generically, the properties of transitivity and chain-transitivity are also equivalent for vector fields. To see why it holds, let $X$ be a closed Riemannian manifold and fix $\SR\subset \mathfrak{X}^1(X)$ be the residual set given in  \cite[Lemma 3.12]{GYZ22}. Fix $V\in \SR$ and assume the flow  $\psi$ generated by $V$ is chain-transitive.  Hence, $CR(\psi)=X$ is the only chain-reurrent class of $\psi$. So, items (3) and (5) of \cite[Lemma 3.2]{GYZ22} give that $X$ is homoclinic class and hence $\psi$ is transitive. Now, the proof finishes by applying Lemma \ref{lemma: non-transitive} in combination with \ref{thmA}.    

Next, assume $\dim(M)=3$ and let $\mathcal{R}=\SR_0\cap R_1$, where $\SR_0$ and $\SR_1$ are, respectively, the residual sets given in Theorems \ref{thm:generic flows} and \ref{thm_residual}. The set $\SR$ is clearly a residual set. Fix $V \in \mathcal{R}$ and assume that its generated flow $\psi$  is not a transitive Anosov flow. Then, according to \ref{thm_residual}, exactly one of the following two cases holds:
\begin{enumerate}
    \item $\psi$ has infinitely many sinks or sources;
    \item $\psi$ is a singular Axiom A flow without cycles.
\end{enumerate}
We shall proceed with the proof by analyzing each case individually.

\medskip
\noindent\textit{\textbf{Case 1:} $\psi$ has infinitely many sinks or sources.}
\medskip

This case is an immediate consequence of Theorem \ref{thmA}. Indeed, $\psi$ has infinitely many sinks of sources, then $CR(\psi)$ cannot be chain-transitive because the chain-recurrent class of any sink or source is trivial.  

\medskip
\noindent\textit{\textbf{Case 2:} $\psi$ is singular Axiom A without cycles}
\medskip

Suppose $\psi$ is a singular Axiom A flow without cycles, and denote its non-wandering set by $\Omega(\psi) = \Lambda_1 \cup \dots \cup \Lambda_n$. First, assume that every $\Lambda_i$ is a hyperbolic basic piece. In this case, $\psi$ is indeed a classical Axiom A flow. Since $\psi$ is not a transitive Anosov flow, then it has least two distinct basic pieces. Furthermore, since $\psi$ has no cycles, at least one of these basic pieces must be an attractor, and another must be a repeller. Consequently, the chain recurrent set $CR(\psi)$ is not chain-transitive, and the result follows from Theorem~\ref{thmA} and Lemma \ref{lemma: non-transitive}.

Now, suppose that at least one of the sets $\Lambda_i$ is a singular hyperbolic set which is not regular hyperbolic (the case where $\Lambda_i$ is a repeller follows by applying the same argument to the flow generated by $-V$). In this scenario, $\Lambda_i \cap \operatorname{Sing}(\psi) \neq \emptyset$, and each singularity in $\Lambda_i$ is hyperbolic. Let $T_{\Lambda_i} M = E^{ss} \oplus E^{cu}$ be the singular hyperbolic splitting over $\Lambda_i$. Since the stable bundle $E^{ss}$ is uniformly contracting, standard invariant manifold theory \cite{HPS} ensures it integrates into a strong stable manifold $W^{ss}(x)$ for each $x \in \Lambda_i$. By \cite{MP03}, if $\sigma \in \Lambda_i$ is a singularity, then $\dim(W^{ss}(\sigma)) = 1$ and $W^{ss}(\sigma) \cap \Lambda_i = \{\sigma\}$. 

Finally, let $x\in W^{ss}(\sigma)\setminus\{\sigma\}$.  Since $\Lambda_i$ is an attractor,  $\alpha(x) \subset (M \setminus U)$, where $U$ is the trapping region of $\Lambda_i$. Otherwise $x\in \Lambda_i$, contradicting the choice of $x$. Consequently, $CR(\psi)$ cannot be chain-transitive, and the proof is again completed by applying Theorem~\ref{thmA} and Lemma \ref{lemma: non-transitive}.

\end{proof}

\subsection{\texorpdfstring{$C^0$}{C0}-dense homeomorphisms}
Let us now concentrate in the Proof of Theorem \ref{thm: dense homeo}. For this sake, we need first prove some approximation and perturbation lemmas. Let us assume $X$ is a compact and coundaryless topological $n$-manifold and denote $LIP(X)$ for the set of Lipschitz continuous homeomorphisms of $X$.  The first ingredient of our proof is a classical result of density of Lipschitz homeomorphisms in topological manifolds which is presented in the following as a consequence of   \cite[Corollary 4.5]{TV82} and \cite[Theorem 4.6]{TV82}.   

\begin{lemma}\label{lemma:LP-manifold}
 If $n\neq 4$, then  $Lip(X)$ is $C^0$-dense in $Homeo(X)$.    
\end{lemma}
The next ingredient is the following lemma whose proof will be omitted, since it can be easily obtained from standard topology arguments (see for instance \cite{Hirsh}).

\begin{lemma}\label{lemma:translation}
    For any $r>0$ and any pair of point $x,y\in B_{\frac{r}{3}}(0)\subset \R^n$ there is a Lipschitz homeomorphism $h:\overline{B_{r}(0)}\to \overline{B_{r}(0)}$ such that: $$h(x)=y \textrm{ and } h|_{\partial B_{r}(0)}=Id_{\partial B_{r}(0)}.$$
\end{lemma}
We now establish the perturbation lemmas that will be used in the sequel. The first one is a variant of the classical $C^0$-closing lemma for homeomorphisms. The overall structure of our proof is similar to the standard one found in the literature. However, since we require the perturbed homeomorphism to be Lipschitz continuous, we need to carefully track the Lipschitz regularity throughout the construction. This regularity issue is often overlooked or left implicit in classical proofs. Therefore, we provide a self-contained proof here, with emphasis on the Lipschitz estimates.

\begin{theorem}\label{lemma:closing}
    For any $f\in Lip(X)$ and any $\eps>0$, there is  $g\in Lip(X)$ such that $g$ has a periodic point and $d_{C^0}(f,g)\leq \eps$. 
\end{theorem}
\begin{proof}
 Let $f \in \operatorname{Lip}(X)$. We may assume without loss of generality that $\operatorname{Per}(f) = \varnothing$; otherwise the conclusion is trivial. Since $X$ is compact, the dynamical system $(X, f)$ admits a recurrent point $x \in X$. Let $V$ be a sufficiently small neighborhood of $x$, and let $n$ be the first positive return time of $x$ to $V$, i.e.,
\[
n := \min \{ k \geq 1 : f^k(x) \in V \}.
\]

Let $\phi : V \to B_r(0) \subset \mathbb{R}^d$ be a Lipschitz continuous coordinate chart such that $\phi(x) = 0$. For each $0 < s \leq r$, define
\[
D(s, p) := \phi^{-1}(B_s(0)).
\]
By shrinking $V$ if necessary, we may assume that $f^n(x) \in D(r/3, p)$. By Lemma \ref{lemma:translation}, there exists a Lipschitz homeomorphism
\[
h_0 : B_r(0) \to B_r(0)
\]
satisfying
\begin{equation} \label{eq:closing}
h_0(\phi(f^n(x))) = \phi(x) = 0
\quad \text{and} \quad
h_0|_{\partial B_r(0)} = \operatorname{Id}|_{\partial B_r(0)}.
\end{equation}

Define
\[
h := \phi^{-1} \circ h_0 \circ \phi : V \to V.
\]
Since $h_0$ and $\phi$ are Lipschitz homeomorphisms (with Lipschitz inverses), $h$ is also a Lipschitz homeomorphism. Moreover, from \eqref{eq:closing} we obtain
\[
h(f^n(x)) = x
\quad \text{and} \quad
h|_{\partial V} = \operatorname{Id}|_{\partial V}.
\]
Consequently, we can extend $h$ continuously to all of $X$ by setting
\[
h|_{X \setminus V} := \operatorname{Id}_{X \setminus V}.
\]
This extension is clearly Lipschitz.

Now define a new map
\[
g := h \circ f : X \to X.
\]
Clearly $g$ is Lipschitz, being a composition of Lipschitz maps. Observe that $f$ and $g$ coincide on $X \setminus f^{-1}(V)$, since $h = \operatorname{Id}$ outside $V$. Moreover, for any $\varepsilon > 0$, by choosing $V$ sufficiently small we can ensure that
\[
d(f(y), g(y)) \leq \varepsilon \quad \text{for all } y \in f^{-1}(V).
\]
This follows from the uniform continuity of $h$ and the fact that $h$ tends to the identity as $V$ shrinks (since $h$ is the identity on $\partial V$). Hence,
\[
d_{C^0}(f, g) \leq \varepsilon.
\]

It remains to show that $x$ is a periodic point for $g$. Since $n$ was chosen as the first positive return time of $x$ to $V$, two cases arise:

\begin{enumerate}
\item If $n = 1$, then $f(x) \in V$, and therefore
\[
g(x) = h(f(x)) = x,
\]
because $h(f(x)) = x$ by the defining property of $h$.

\item If $n \geq 2$, then for every $i = 1, \dots, n-1$, we have $f^i(x) \notin V$ by the minimality of $n$. In particular,
\[
f^i(x) \notin f^{-1}(V) \quad \text{for } i = 1, \dots, n-2,
\]
 Thus,
\[
g^i(x) = f^i(x) \quad \text{for } i = 1, \dots, n-1,
\textrm{ and }
g^n(x) = h(f^n(x)) = x.
\]
 
\end{enumerate}

Hence $x$ is periodic of period $n$ for $g$ and  the proof is complete.
\end{proof}

We continue with our perturbation lemmas. At this point, it is worth clarifying an important subtlety in the proof of Theorem \ref{thm: dense homeo}. Our strategy relies on applying Theorem \ref{thmA}, which is stated for homeomorphisms with finite topological entropy. Therefore, in order to obtain the desired density result, we must ensure that the approximating homeomorphisms we construct have finite topological entropy. This is a nontrivial issue, as the $C^0$-topology is extremely flexible: indeed, infinite entropy is $C^0$-generic, meaning that an arbitrarily small $C^0$-perturbation of a given Lipschitz homeomorphism can cause its entropy to explode. Consequently, our perturbations must be carried out carefully to avoid entropy explosion at every step. The next theorem  asserts that we can perturb a homeomorphism with a fixed point by blowing it up into an entire neighborhood of fixed points, while keeping the topological entropy finite.    

\begin{theorem}\label{lemma:explosion}
Let $X$ be a compact manifold of dimension $n$, $f: X\to X$ a Lipschitz, and $p \in X$ a fixed point of $f$. Then, for every $\varepsilon > 0$ there exist an  open neighborhood $U$ of $p$ and a homeomorphism $\tilde{f}: X \to X$ satisfying:
\begin{enumerate}
    \item There exists a closed ball $B \subset U$ centered at $p$ such that $\tilde{f}|_B = \mathrm{id}_B$;
    \item $d_{C^0}(\tilde{f}, f) < \varepsilon$;
    \item $h_{\mathrm{top}}(\tilde{f}) \leq h_{\mathrm{top}}(f) < \infty$.
\end{enumerate}
\end{theorem}

\begin{proof}
    Let $X$ be a $n$-dimensional manifold and suppose $f:X\to X$ is as in the statement.  We will use local coordinates to construct  $\tilde f$ in a neighborhood of $p$.  For this sake, consider a  coordinate chart $\phi: V \to U\subset \mathbb{R}^n$ with $p \in V$ and $\phi(p) = 0$.  
    Without loss of generality, we may assume $B_r(0) \subset \phi(V)$,  for some sufficiently small $r > 0$. For simplicity,  we denote $D(r,p)=\phi^{-1}(B_r(0))$.  

Since $f$ is continuous and $f(p) = p$, we can assume  $f(D(r,p)) \subset D(2r,p))$. Let $\rho: [0, \infty) \to [0, 1]$ be an increasing smooth function such that:
\begin{equation}
    \rho(t) = \begin{cases}
        0 & \text{if } t \leq \frac{r}{2}, \\
        1 & \text{if } t \geq r.
    \end{cases}
\end{equation}
Now, we define a function $\tilde f=\phi^{-1}\circ g_0:D(r,p)\to D(2r,p)$, where $\tilde f_0:D(r,p)\to B_{2r}(0)$ is given by 
\begin{equation}
    g(x) = (1 - \rho(\Vert \phi(x)\Vert)) \cdot \phi(x) + \rho(\Vert \phi(x)\Vert)) \cdot \phi(f(x)).
\end{equation}
Observe that by construction it holds
\begin{enumerate}
    \item $\tilde f(x)=x$, for every $x\in D(\frac{r}{2},p)$.
    \item $\tilde f(x)=f(x)$, for every $x\in \partial D(r,p)$.
\end{enumerate}
Moreover, the action of $\tilde f$ on the region $D(r,p) \setminus D(\tfrac{r}{2},p)$ can be described as follows. For every radius $t$ satisfying $\tfrac{r}{2} < t < r$ and for every point $x$ on the sphere $\phi^{-1}(S_t(0))$, there exists a curve 
\[
\gamma_x(s) = \phi^{-1}\big((1-s)x + s\phi(f(x))\big), \quad s \in [0,1],
\]
which connects $x$ to $f(x)$. Furthermore, for a fixed $t$, the curves $\gamma_x$ are pairwise disjoint. The homeomorphism $\tilde f$ then sends each point $x$ to the intersection point $\gamma_x(t)$ of the corresponding curve with the sphere of radius $t$. From this construction, we conclude that $\tilde f$ is indeed a homeomorphism. Finally, extend $\tilde f$ to $X$, by setting $\tilde f(x)=f(x)$, for every $x\in X\setminus D(r,p)$. It is clear that $\tilde f\in Homeo(X)$.

We claim that $\tilde f$ satisfies the desired properties. To verify this, first note that, given any $\varepsilon > 0$, the continuity of $f$ ensures the existence of a constant $0 < r < \frac{\varepsilon}{4}$ such that 
\[
f(D(r,p)) \subset B_{\frac{\varepsilon}{4}}(D(r,p)).
\]
Consequently, for every $x \in D(r,p)$, the construction yields
\[
d(f(x), g(x)) \leq d(f(x), x) + d(x, g(x)) \leq \frac{\varepsilon}{2} + \frac{\varepsilon}{2} = \varepsilon.
\]
Since $f \equiv g$ on $X \setminus D(r,p)$, we conclude that 
\[
d_{C^0}(f,g) \leq \varepsilon.
\]

Finally, observe that the topological entropy of $\tilde f$ satisfies
\[
h_{\mathrm{top}}(\tilde f) = \max\left\{ h_{\mathrm{top}}(\tilde f, M\setminus D(\tfrac{r}{2})),\; h_{\mathrm{top}}(\tilde f, \overline{D(\tfrac{r}{2})}) \right\}.
\]
However, because $\tilde f|_{\overline{D(\frac{r}{2})}} = \mathrm{Id}_{\overline{D(\frac{r}{2})}}$, we have 
\[
h_{\mathrm{top}}(\tilde f, \overline{D(\tfrac{r}{2})}) = 0.
\]
On the other hand, since $f$ is Lipschitz and both $\phi$ and $\rho$ are smooth, it follows that $\tilde f$ is Lipschitz on $M \setminus D(\tfrac{r}{2})$. Therefore, its topological entropy on that region is finite, and hence $h_{\mathrm{top}}(\tilde f) < \infty$. This completes the proof.

\end{proof}

Our last ingredient, that is presented in the sequel, allow us to perform an extra perturbation to crate sinks, while preserving finiteness of topological entropy.

\begin{lemma}\label{lema:sinkcreation}
Let $\tilde{f}: X \to X$ be the homeomorphism constructed in Theorem \ref{lemma:explosion}. Then, for every $\varepsilon > 0$, there exists a homeomorphism $g: X \to X$ satisfying:
\begin{enumerate}
    \item $p$  sink for $g$;
    \item $d_{C^0}(g, \tilde{f}) < \varepsilon$;
    \item $h_{\mathrm{top}}(g) \leq h_{\mathrm{top}}(\tilde{f}) < \infty$.
\end{enumerate}
\end{lemma}
\begin{proof}
The proof follows the same strategy as that of Theorem~\ref{lemma:explosion}. As before, let $\phi: V \to U$ be a local chart centered at $p$, i.e., $\phi(p)=0$, and assume that $D(r,p) = \phi^{-1}(B_r(0)) \subset V$ consists entirely of fixed points of $\tilde f$.

Choose a smooth decresing function $\eta: [0, \infty) \to [0, 1]$ satisfying
\begin{equation}
    \eta(t) = 
    \begin{cases}
        1, & \text{if } t \leq \frac{r}{4}, \\
        0, & \text{if } t \geq \frac{r}{2}.
    \end{cases}
\end{equation}

Next, define a contraction $h: B_{\frac{r}{4}}(0) \to B_{(1-\delta)\frac{r}{4}}(0)$ by $h(x) = (1-\delta)x$, with $0 < \delta < 1$. We then set $g = \phi^{-1} \circ g_0$, where
\[
g_0(x) = \eta(\| \phi(x) \|) \, h(\phi(x)) + \big(1 - \eta(\| \phi(x) \|)\big) \, \phi(\tilde f(x)).
\]
For points outside $D(r,p)$, we define $g(x) = \tilde f(x)$ for every $x \in X \setminus D(r,p)$.

Clearly, $g \in \mathrm{Homeo}(X)$. Moreover, the contraction $h$ renders $p$ a sink for $g$, since points in $D(\frac{r}{4},p)$ are attracted to $p$.

As in Theorem~\ref{lemma:explosion}, it is straightforward to verify that, for any given $\varepsilon > 0$, we may choose $r, \delta > 0$ sufficiently small so that
\[
d_{C^0}(\tilde f, g) < \varepsilon.
\]
Since $D(\frac{r}{4},p)$ is contained in the basin of attraction of $p$ for $g$, this region contributes nothing to the topological entropy of $g$. On the other hand, by construction, $g$ is Lipschitz on $X \setminus D(\frac{r}{4},p)$; hence its topological entropy there is finite. Consequently,
\[
h_{\mathrm{top}}(g) < \infty,
\]
which completes the proof.

\end{proof}
\begin{remark}\label{rmk:fixedperiodic}
   We observe that, although the previous lemmas are stated for fixed points, they also hold for periodic points. More precisely, if $p$ is a periodic point of period $k$, then Theorem \ref{lemma:explosion} blows up $p$ into an open neighborhood consisting entirely of periodic points of period $k$, while Lemma \ref{lema:sinkcreation} produces a periodic sink of period $k$.
      The proofs in the fixed point case were chosen for expositional simplicity, as they capture all the essential ideas while avoiding the heavier notation that would accompany the periodic setting. Adapting them to periodic points presents no conceptual difficulty and can be done in a routine manner.

\end{remark}

We are finally in a position to provide te reader with a proof for Theorem \ref{thm: dense homeo}.

\begin{proof}[Proof of Theorem \ref{thm: dense homeo}]
Let $X$ be a compact and boundaryless topological manifold of dimension different from $4$. By Lemma \ref{lemma:LP-manifold} we have that $Lip(X)$ is  $C^0$-dense in $Homeo(X)$. Therefore, given any homeomorphism $f_0 \in \mathrm{Homeo}(X)$, there exists a  $f_1\in Lip(X)$ that is arbitrarily $C^0$-close to $f_0$. Moreover, by Lemma \ref{lemma:closing}, we may assume without loss of generality that $f_1$ has a periodic point. Recall that being Lipschitz, we have $h_{top}(f_1)< \infty$. 

Now, applying Theorem \ref{lemma:explosion} and observing Remark \ref{rmk:fixedperiodic} , we obtain a homeomorphism $f_2$ arbitrarily close to $f_1$ in the $C^0$-topology, which retains finite topological entropy and features a closed ball consisting entirely of periodic points with same period. Subsequently, Lemma \ref{lema:sinkcreation} and again Remark \ref{rmk:fixedperiodic} yields a homeomorphism $f$ that is $C^0$-close to $f_2$, also with finite topological entropy, and possesses a strict attracting periodic point. 

Consequently, $f$ is not chain-transitive. By Theorem \ref{thmA}, this implies that $f$ admits a H\"older continuous potential displaying a phase transition. Since all the perturbations performed in this construction can be made arbitrarily small, we conclude that there exists a dense subset of homeomorphisms in $\mathrm{Homeo}(X)$ with finite topological entropy for which there exists a H\"older continuous potential exhibiting a phase transition. This concludes the proof.

\end{proof}

\section*{Acknowledgments.}
The authors are grateful for Professors Alejandro Kocsard and Paulo Varandas whose valuable suggestions  helped us to refine this work. We acknowledge the CT UFRJ hot‑dog stand, where the quiet seed of this work was first sown.
Alexander Arbieto was partially supported by CAPES– Finance Code 001, CNPq Grant 307877/2025-6, PRONEX-Dynamical Systems and FAPERJ “Programa Cientista do Nosso Estado", E-26/201.181/2022 and E-26/200.281/2026. Walter Britto was partially supported by CAPES from Brazil.  
\begin{table}[h]
\begin{tabularx}{\linewidth}{p{1.5cm}  X}
\includegraphics [width=1.8cm]{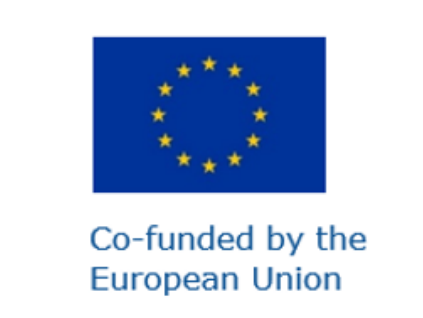} &
\vspace{-1.5cm}
This research is part of a project that has received funding from
the European Union's European Research Council Marie Sklodowska-Curie Project No. 101151716 -- TMSHADS -- HORIZON--MSCA--2023--PF--01.\\
\end{tabularx}
\end{table}

\end{document}